\documentclass[12pt,reqno]{amsart}
\usepackage[a4paper,bindingoffset=0.1in,left=1.2in,right=1.3in,top=1.4in,bottom=1.5in,footskip=.25in]{geometry}
\usepackage{indentfirst,amssymb,amsmath,amsthm}    
\usepackage{newtxtext,newtxmath}
\usepackage{setspace}
\usepackage{times}
\usepackage[utf8]{inputenc}
\usepackage[T1]{fontenc}
\usepackage{verbatim}
\usepackage{hyperref}
\hypersetup{colorlinks=true, linkcolor=blue,citecolor=blue, urlcolor=blue}
\newcommand{\divides}{\mid}
\newcommand{\notdivides}{\nmid}
\DeclareMathOperator{\ch}{ch}

\DeclareMathOperator{\ord}{ord}
\DeclareMathOperator{\dime}{dim}
\DeclareMathOperator{\pideg}{PI-deg}

\DeclareMathOperator{\gcdi}{gcd}
\DeclareMathOperator{\lcmu}{lcm}

\DeclareMathOperator{\tors}{tor}

\DeclareMathOperator{\mo}{mod}
\usepackage{mathtools}

\numberwithin{equation}{section}
 
\newtheorem{theo}{Theorem}[section]
\newtheorem{defi}{Definition}[section]
\newtheorem{prop}{Proposition}[section]
\newtheorem{coro}{Corollary}[section]
\newtheorem{lemm}{Lemma}[section]
\newtheorem{rema}{Remark}[section]

\begin{document}

\setcounter{page}{1} 
\baselineskip .65cm 
\pagenumbering{arabic}

\title[Quantum Heisenberg Enveloping Algebra]{Quantum Heisenberg Enveloping Algebras\\ at roots of unity}
\author [S. Bera,\ \  S. Mandal,\ \ S. Nandy]{Sanu Bera\ \ \ Sugata Mandal \ \ \ Soumendu Nandy}

\address {\newline Sanu Bera$^1$\ \ \ Sugata Mandal$^2$ \ \ \ Soumendu Nandy$^3$\newline School of Mathematical Sciences, \newline Ramakrishna Mission Vivekananda Educational and Research Institute (rkmveri), \newline Belur Math, Howrah, Box: 711202, West Bengal, India.}
\email{\href{mailto:sanubera6575@gmail.com}{sanubera6575@gmail.com$^1$};\href{mailto:gmandal1961@gmail.com}{gmandal1961@gmail.com$^2$};\newline \href{mailto:Soumendu.nandy100@gmail.com}{Soumendu.nandy100@gmail.com$^3$}}

\begin{abstract}
%In this article, the two-parameter analogs of Heisenberg enveloping algebra
In this article, the two-parameter quantum Heisenberg enveloping algebra, which serves as a model for certain quantum generalized Heisenberg algebras, have been studied at roots of unity. In this context, the quantum Heisenberg enveloping algebra becomes a polynomial identity algebra, and the dimension of simple modules is bounded by its PI degree. The PI degree, center, and complete classification of simple modules up to isomorphism are explicitly presented. We work over a field of arbitrary characteristic, although our results concerning the representations require that it is algebraically closed.
%our results on the representations require that it be algebraically closed.
\end{abstract}
\maketitle
%\tableofcontents
\section{{Introduction}}
Let $\mathbb{k}$ be a field and $\mathbb{k}^*:=\mathbb{k}\setminus\{0\}$. The three-dimensional Heisenberg Lie algebra $\mathfrak{h}$ over $\mathbb{k}$ is defined on the basis $\{x,y,z\}$ subject to the relations $[x,y]=z,[x,z]=0,[y,z]=0$. The algebra $\mathfrak{h}$ is a nilpotent Lie algebra that occurs in quantum mechanics in the solution of the harmonic oscillator problem. The first Weyl algebra $A_1(\mathbb{k}):=\mathbb{k}\langle x,y~|~xy-yx=1\rangle$ is a simple factor ring $U(\mathfrak{h})/\langle z-1\rangle$ of the enveloping algebra $U(\mathfrak{h})$ and its oscillator representation gives the solution to the harmonic oscillator problem (cf. \cite{ks}). Determining the simple factor rings for enveloping algebras and quantum enveloping algebras is a long-standing problem in algebra and representation theory. %Algebraic methods have long been applied to solutions of a large number of physical systems and many physical phenomena. The most classical example of an algebraically solved system is the harmonic oscillator whose underlying algebra is Heisenberg algebra, generated by annihilation and creation operators (cf. \cite[Section 1]{ks}). More precisely, the oscillator representation of the Weyl algebra $A_1(\mathbb{k})\cong U(\mathfrak{h}/\langle z-1\rangle$ gives the solution to the harmonic oscillator problem.
%Let $\mathbb{k}$ be a field and $\mathbb{k}^*:=\mathbb{k}\setminus\{0\}$. The three dimensional Heisenberg Lie algebra $\mathfrak{h}$ over $\mathbb{k}$ is defined on the $\mathbb{k}$-basis $\{x,y,z\}$ subject to the relations $[x,y]=z,[x,z]=0,[y,z]=0$. The first Weyl algebra $A_1(\mathbb{k}):=\langle x,y~|~xy-yx=1\rangle$ is a simple factor ring of the enveloping algebra $U(\mathfrak{h})$.  Determining the simple factor rings for enveloping algebras and quantum enveloping algebras is a long-standing problem in algebra and representation theory. Algebraic methods have long been applied to solutions of a large number of physical systems and many physical phenomena. The most classical example of an algebraically solved system is the harmonic oscillator whose underlying algebra is Heisenberg algebra, generated by annihilation and creation operators (cf. \cite[Section 1]{ks}). More precisely, the oscillator representation of the Weyl algebra $A_1(\mathbb{k})\cong U(\mathfrak{h}/\langle z-1\rangle$ gives the solution to the harmonic oscillator problem.
\par In \cite{ks}, Kirkman and Small studied a $q$-analog to $U(\mathfrak{h})$, known as quantum Heisenberg enveloping algebra. % where the commutator is replaced by the “quommutator”, i.e., $[a,b]_q=ab-qba$. 
For $q\in\mathbb{k}^*$, the quantum Heisenberg enveloping algebra $\mathcal{H}_q$ is the $\mathbb{k}$-algebra generated by $x,y,z$ subject to the relations
\[zx=q^{-1}xz,\ zy=qyz,\ yx-qxy=z.\]
This algebra arises from the work of physicists in constructing a $q$-analog to the quantum Harmonic oscillator (cf. \cite{am,th}). The algebra $\mathcal{H}_q$ is an iterated skew polynomial ring over $\mathbb{k}[z]$ % extension of the quantum plane $\mathcal{O}_q(\mathbb{k}^2)=\langle x,z~|~xz=qzx\rangle$ 
and hence this is a prime affine Noetherian domain \cite{ks}. A key result in \cite{ks} is that there exists a central element $\Omega:=(yx-q^{-1}xy)z$ in $\mathcal{H}_q$ such that $\mathcal{H}_q/\langle\Omega-1\rangle$ is isomorphic to Hayashi’s $q$-analog of the Weyl algebra, hence providing a simple factor ring in the quantum setting.
\par In \cite{jg}, Jasson Gaddis introduced and studied a two-parameter analog of Heisenberg enveloping algebra $U(\mathfrak{h})$. For $p,q\in\mathbb{k}^*$, the quantum Heisenberg enveloping algebra $\mathcal{H}_{p,q}$ is the $\mathbb{k}$-algebra generated by $x,y,z$ subject to the relations
\[zx=p^{-1}xz,~ zy=pyz,~yx-qxy=z.\]
Notice that when $p=q$, the algebra $\mathcal{H}_{p,q}$ becomes $\mathcal{H}_{q}$ and hence $\mathcal{H}_{p,q}$ generalizes the one parameter analog $\mathcal{H}_q$. The algebra $\mathcal{H}_{p,q}$ has an iterated skew polynomial presentation twisted by automorphisms and derivation (see Subsection \ref{basic}) and is a conformal ambiskew polynomial ring as defined \cite{dj,dj1}. In \cite{jg}, the author studied prime ideals in $\mathcal{H}_{p,q}$, and related algebras. %In the independent parameter case: $p^r\neq q^s$ for all $r,s\in \mathbb{Z}\setminus\{0\}$, the multiparameter Weyl algebra of G. Benkart obtained from $\mathcal{H}_{p,q}$ through the use of noncommutative dehomogenization. 
In the dependent parameter case assuming $p$ and $q$ are not roots of unity, the algebra $\mathcal{H}_{p,q}$ has a simple factor ring %$\mathcal{A}_p(r,s)$ 
which generalizes the Hayashi-Weyl algebra (see \cite[Theorem 6.5]{jg}). %as $\mathcal{A}_p(1,1)=\mathcal{A}_q$.
In the inverse parameter case i.e., when $p=q^{-1}$, denote this algebra by $\mathcal{H}'_q$. This algebra was studied in \cite{ks,jg} and they showed that every nonzero prime ideal of $\mathcal{H}'_q$ contains $z$ (see \cite[Proposition 4.1]{jg}).
\par Most of the above-studied results for $\mathcal{H}_{p,q}$ are in generic cases when the deformation parameters are not roots of unity. % (equivalently $\mathcal{H}_{p,q}$ is not PI). 
In \cite{dj1} Jordan accomplished the classification of finite-dimensional simple modules over the conformal ambiskew polynomial rings $A[y,\alpha][x,\alpha^{-1},\delta]$ of automorphism type, where $A$ is an affine domain over an algebraically closed field. Recently in \cite{lk2}, Lopes and Razavinia introduced a new class of algebras $\mathcal{H}_q(f,g)$ for $f,g\in \mathbb{k}[z]$, named quantum generalized Heisenberg algebra, given by
\[\mathcal{H}_q(f,g):=\mathbb{k}\langle z,x,y~|~zx=xf(z),yz=f(z)y,yx-qxy=g(z)\rangle.\] This algebra $\mathcal{H}_q(f,g)$ is an ambiskew polynomial ring of endomorphism type over $\mathbb{k}[z]$ which includes the $\mathcal{H}_{p,q}$ for $f(z)=pz,g(z)=z$ and includes the generalized Heisenberg algebras for $g(z)=f(z)-z$ as studied in \cite{rk}. They explored the algebraic structure, classified finite-dimensional representations \cite{lk2}, and solved the isomorphism problem \cite{lk} for $\mathcal{H}_q(f,g)$.
%It's worth noting that a specific quantum generalized Heisenberg algebra exhibits isomorphism with $\mathcal{H}_{p,q}$.
\par Let $\alpha$ be a $\mathbb{k}$-algebra automorphism of $\mathbb{k}[z]$. Then $\alpha$ is linear and of the form $\alpha(z)=\omega z+b$ for some $0\neq\omega,b\in\mathbb{k}$. The order of $\alpha$ in the automorphism group of $\mathbb{k}[z]$ is finite, say $l$ if and only if $\alpha(z)=\omega z+(1-\omega)b$ for some $l$-th primitive root $\omega$ of unity and $b\in\mathbb{k}$. For such an $\alpha$ with $\omega\neq q$, it follows from \cite[Proposition 2.2]{lk} that the quantum generalized Heisenberg algebra $\mathcal{H}_q(f,g)$ for $f(z)=\alpha(z)$ and $g(z)=\alpha(z)-qz$ is isomorphic to the $\mathcal{H}_{p,q}$ for $p=\omega^{-1}$. Hence the algebra $\mathcal{H}_{p,q}$ with $pq\neq 1$ serves as a model for certain quantum generalized Heisenberg algebras at roots of unity. However the analysis of the ambiskew polynomial algebra $\mathcal{H}_{p,q}$ and its representation theory, throughout the research literature including \cite{dj1,lk2}, has been rather abstract. 
\par In this article, we wish to study the algebra $\mathcal{H}_{p,q}$ and its representations in the case when $p$ and $q$ are roots of unity.
%, using linear algebra. %Typically, when the parameters $p$ and $q$ are specialized to roots of unity, 
In the roots of unity context, the algebra $\mathcal{H}_{p,q}$ with $pq\neq 1$ becomes a finitely generated module over its center and hence a polynomial identity (PI) algebra. The theory of PI-algebras is a crucial tool for studying this algebra. The PI degree is a fundamental invariant that bounds the $\mathbb{k}$-dimension of simple $\mathcal{H}_{p,q}$-modules, and this bound is in fact attained (see Subsection \ref{pialg}). The computation of the PI degree for $\mathcal{H}_{p,q}$ is of substantial importance as it sheds light on the classification of simple $\mathcal{H}_{p,q}$-modules. In the inverse parameter case, the algebra $\mathcal{H}'_q$ is primitive and hence not a PI algebra when $\ch(\mathbb{k})=0$. Moreover, the algebra $\mathcal{H}'_q$ becomes a PI algebra if and only if $\ch(\mathbb{k})>0$ (see Section \ref{sec9}).
%The theory of PI-algebras is a crucial tool to study this algebra, with the PI degree serving as a fundamental invariant that bounds the $\mathbb{k}$-dimension of simple $\mathcal{H}_{p,q}$-modules and this bound will be attained.
%the minimal degree and PI degree serving as fundamental invariants of the PI-algebra. In the case of a prime affine PI-algebra over an algebraically closed field, the PI degree bounds the $\mathbb{k}$-dimension of a simple module \cite[Theorem I.13.5]{brg} and this bound will be attained by the algebra \cite[Lemma III.1.2]{brg}. %Hence the computation of the PI degree of a quantized PI-algebra bears substantial importance. This paper aims to derive a precise formula for the PI degree of both versions of quantized Weyl algebras at roots of unity.
%\subsection*{Assumptions}
\par In the roots of unity context, we shall assume that $p$ and $q$ are primitive $m$-th and $n$-th roots of unity in $\mathbb{k}$, respectively. We will maintain the assumption $pq\neq 1$ throughout this article, except in Section \ref{sec9}. This condition ensures that neither $m$ nor $n$ can simultaneously equal 1.
%This condition guarantees that both $m$ and $n$ cannot be equal to $1$ simultaneously. 
%Throughout the article $\mathbb{k}$ is an algebraically closed field of arbitrary characteristics and all modules are the right modules. 
We work over a field $\mathbb{k}$ of arbitrary characteristics but require it to be algebraically closed for our results on representations. Throughout the article, all modules under consideration are right modules.
%\subsection*{Arrangements}
\par The present paper is organized as follows. In Section \ref{sec2}, we begin by revisiting essential properties of $\mathcal{H}_{p,q}$ and laying the groundwork with necessary facts for the theory of Polynomial Identity algebras.
%we recall some basic properties of $\mathcal{H}_{p,q}$ and discuss some necessary facts for the theory of Polynomial Identity algebras. delve into explicit calculations of the PI degree
In Section \ref{pisection}, we explicitly compute the PI degree of $\mathcal{H}_{p,q}$ using derivation erasing and the PI degree of quantum affine space. In Section \ref{sec4}, we compute the center of $\mathcal{H}_{p,q}$ completely. In Section \ref{sec5}, we construct three types of $z,\theta$-torsionfree simple $\mathcal{H}_{p,q}$-modules: $\mathcal{V}_1(\mu,\lambda,\gamma)$, $\mathcal{V}_2(\mu,\lambda)$ and $\mathcal{V}_3(\lambda)$ for any $\mu,\lambda,\gamma\in\mathbb{k}^*$. In Section \ref{sec6}, we proceed to classify all simple $\mathcal{H}_{p,q}$-modules based on the action of certain normal and central elements within $\mathcal{H}_{p,q}$. %More precisely, there are three types of $z,\theta$-torsionfree simple modules: 
In Section \ref{sec7}, we then determine the isomorphism classes of simple $z,\theta$-torsionfree modules. %In section 8, we identify $m$ and $n$ for which the $\mathbb{k}$-dimension of $\mathcal{V}_3(\lambda)$ that is equal to $\ord(pq)$ is always maximum or non-maximum.
%Finally, we conclude our study by determining the values of $m$ and $n$ for which the $\mathbb{k}$-dimension of $\mathcal{V}_3(\lambda)$, which equals $\ord(pq)$, consistently resides either at its maximum or remains non-maximal.
In Section \ref{maximum}, we determine the values of $m$ and $n$ such that the $\mathbb{k}$-dimension of $\mathcal{V}_3(\lambda)$, which equals $\ord(pq)$, is consistently either at its maximum or non-maximum. In the last section, we will consider the inverse parameter case.
\section{Preliminaries}\label{sec2}
\subsection{Torsion and Torsionfree Modules} 
Let $A$ be an algebra and $M$ be a right $A$-module and $S\subset A$ be a right Ore set. The submodule
\[\tors_{S}(M):=\{m\in M~|~ms=0\ \text{for some}\  s\in S\}\]
is called the $S$-torsion submodule of $M$. The module $M$ is said to be $S$-torsion if $\tors_{S}(M)=M$ and $S$-torsion free if $\tors_{S}(M)=0$. If Ore set $S$ is generated by $x\in A$, we simply say that the $S$-torsion/torsionfree module $M$ is $x$-torsion/torsionfree.
\par A nonzero element $x$ of an algebra ${A}$ is called a {\it normal} element if $x{A}={A}x$. Clearly, if $x$ is a normal element of $A$, then the set $S=\{x^i~|~i\geq 0\}$ is an Ore set generated by $x$. The next lemma is obvious.
\begin{lemm}\label{itn}
Suppose that $A$ is an algebra, $x$ is a normal element of $A$ and $M$ is a simple $A$-module. Then either $Mx=0$ (if $M$ is $x$-torsion) or the map $x_{M}:M\rightarrow M, m\mapsto mx$ is an isomorphism (if $M$ is $x$-torsionfree).
\end{lemm}
The above lemma says that the action of a normal element on a simple module is either trivial or invertible.
\subsection{Identities and Basic Properties}\label{basic} 
Given a matrix $\Lambda=(\lambda_{ij})$ with $\lambda_{ii}=1$ and $\lambda_{ij}=\lambda_{ji}^{-1}$, the quantum affine space is the $\mathbb{k}$-algebra $\mathcal{O}_{\Lambda}(\mathbb{k}^n)$ generated by the variables $x_1,\cdots ,x_n$ subject to the relations \[x_ix_j=\lambda_{ij}x_jx_i, \ \ \ \forall\ \ \ 1 \leq i,j\leq n.\]
The $\mathbb{k}$-subalgebra $R$ of $\mathcal{H}_{p,q}$ generated by $z$ and $x$ is a quantum plane \[\mathcal{O}_p(\mathbb{k}^2)=\langle x,z~|~xz=pzx\rangle\] and can be expressed as $R=\mathbb{k}[z][x,\alpha]$
where $\alpha$ is the automorphism of $\mathbb{k}[z]$ defined by $\alpha(z)=pz$. Therefore the algebra $\mathcal{H}_{p,q}$ can be presented as an iterated skew polynomial ring of the form $R[y,\beta,\delta]$ where $\beta$ is the automorphism of $R$ given by $\beta(z)=p^{-1}z,~\beta(x)=qx$ and $\delta$ is a $\beta$-derivation on $R$ given by $\delta(z)=0,~\delta(x)=z$. Thus the algebra $\mathcal{H}_{p,q}$ is a conformal ambiskew polynomial ring over $\mathbb{k}[z]$. The skew polynomial version of the Hilbert-Basis theorem (cf. \cite[Theorem 2.6]{mcr}) yields the following:
\begin{prop}\label{basish}
The algebra $\mathcal{H}_{p,q}$ is a prime affine Noetherian domain. Moreover, the monomials $z^ix^jy^k$ ($i,j,k\in\mathbb{Z}_{\geq 0}$) form a $\mathbb{k}$-basis for $\mathcal{H}_{p,q}$.
\end{prop}
%The algebra $\mathcal{H}_{p,q}$ has the form of an ambiskew polynomial ring $(A,\alpha,u,q)$, as defined \cite{dj,dj1}, where $A=\mathbb{k}[z],\ u=(1-pq)^{-1}z\in A$ and $\alpha(z)=pz$ is an automorphism of $A$. Moreover $\mathcal{H}_{p,q}$ is isomorphic to generalized Weyl algebra $D(\rho,yx)$, as defined \cite{bj}, where $D=\mathbb{k}[yx,z]$ with $\rho\in \auto(D)$ given by $\rho(z)=pz,\rho(yx)=q^{-1}(yx-z)$. 
The defining relations of $\mathcal{H}_{p,q}$ state that the element $z$ is normal within $\mathcal{H}_{p,q}$. %Furthermore the factor algebra $\mathcal{H}_{p,q}/\langle z\rangle$ is isomorphic to the quantum plane $\mathcal{O}_q(\mathbb{k}^2)$. 
Define $\theta:=yx-p^{-1}xy$. Then we can easily check that 
\begin{itemize}
    \item[(1)] $\theta=(q-p^{-1})xy+z=(1-p^{-1}q^{-1})yx+p^{-1}q^{-1}z$,
    \item[(2)] $\theta z=z\theta,\ \theta x=qx\theta,\ \theta y=q^{-1}y\theta.$
\end{itemize}

Thus $\theta$ is also normal in $\mathcal{H}_{p,q}$. %and the factor algebra $\mathcal{H}_{p,q}/\langle \theta \rangle$ is isomorphic to the quantum plane $\mathcal{O}_p(\mathbb{k}^2)$. 
These two normal elements will play a significant role in the classification of simple modules over $\mathcal{H}_{p,q}$.
%play a significant role in the classification of simple modules.
\par Define the $(p,q)$-number to be
\begin{equation}\label{pqnum}
    [k]_{p,q}:=\frac{q^k-p^{-k}}{q-p^{-1}}=\sum\limits_{i=0}^{k-1}q^ip^{-(k-1-i)}.
\end{equation}
Note that $[k]_{p,q}=0$ if $p$ and $q$ are roots of unity and both orders divide $k$. %Moreover, $[k]_{p,q}=0$ if and only if $k$ divides the $\ord(pq)$.
\begin{lemm} \emph{(\cite[Lemma 3.8]{jg})}\label{identity}
  In $\mathcal{H}_{p,q}$, the following identities hold for $k\geq 0$:
\begin{itemize}
    \item [(1)] $yx^k=q^kx^ky+[k]_{p,q}x^{k-1}z$.
    \item [(2)] $y^kx=q^kxy^k+ [k]_{p,q}zy^{k-1}$.
\end{itemize}
\end{lemm}
\begin{rema}
   It is important to mention that all the aforementioned results for $\mathcal{H}_{p,q}$ are valid for any $p,q\in \mathbb{k}^*$.
\end{rema}
\begin{coro}\label{cp1}
    If $p$ and $q$ are $l$-th roots of unity, then the elements $z^l,x^l$, and $y^l$ are contained in the center of $\mathcal{H}_{p,q}$. 
\end{coro}
\subsection{Polynomial Identity Algebras}\label{pialg} A ring $R$ is said to be a polynomial identity (PI) ring if there exists a polynomial $f(x_1,\cdots,x_k)\in \mathbb{Z}\langle x_1,\cdots,x_k\rangle$ such that $f(r_1,\cdots,r_k)=0$ for all $r_i\in R$.  The minimal degree of a PI ring $R$ is the least degree of all monic polynomial identities for $R$. PI rings cover a large class of rings including commutative rings. Commutative rings satisfy the polynomial identity $x_1x_2-x_2x_1$ and therefore have minimal degree $2$. 
\par In the roots of unity case, it follows from Corollary \ref{cp1} that the algebra $\mathcal{H}_{p,q}$ is a finitely generated module over its center. Hence it is a PI algebra by \cite[Corollary 13.1.13]{mcr}. This roots of unity assumption is also necessary, as explained below.
\begin{prop}\emph{(\cite[Proposition 3.11]{jg})}
    The algebra $\mathcal{H}_{p,q}$ is a PI algebra if and only if $p$ and $q$ are roots of unity.
\end{prop}
Now we define the PI degree of prime PI-algebras. This definition will suffice because the algebras covered in this article are all prime. Due to the Artin-Wedderburn Theorem, any central simple algebra $A$ is isomorphic to a matrix ring over a central simple division ring. Hence $\dime_{Z(A)}(A)=n^2$ for some natural number $n$. From this, we define the PI degree of $A$ to be $n$. We now recall one of the fundamental results from the Polynomial Identity theory (cf. {\cite[Theorem 13.6.5]{mcr}}).
\begin{theo}[Posner’s Theorem]\label{pos}
 Let $A$ be a prime PI ring with center $Z(A)$ and minimal
degree $d$. Let $S=Z(A)\setminus\{0\}$, $Q=AS^{-1}$ and $F=Z(A)S^{-1}$. Then $Q$ is a central simple algebra with center $F$ and $\dime_{F}(Q)=(\frac{d}{2})^2$.
\end{theo} 
Note that the $Q$ in Posner’s theorem is PI and, since $Q$ is a central simple algebra, we can state its PI degree to be $\ \frac {d}{2}$ by the discussion above. Furthermore, as a result, \cite[I.13.2(6)]{brg}, $Q$ has the same minimal degree as $A$, namely $d$. Recognizing that the PI degree can be interpreted as some measure of how close to being commutative a PI-algebra is and that this, in turn, is related to its minimal degree, the definition of PI degree given above can be extended to all prime PI rings in the following way. 
\begin{defi}
The PI degree of a prime PI ring $A$ with minimal degree $d$ is $\pideg(A)=\frac{d}{2}$.
\end{defi}
Kaplansky’s Theorem (cf. \cite[ Theorem 13.3.8]{mcr}) has a striking consequence in the case
of a prime affine PI algebra over an algebraically closed field. The following result provides an
important link between the PI degree of a prime affine PI algebra over an algebraically closed
field and the $\mathbb{k}$-dimension of its irreducible representations (cf. \cite[Theorem I.13.5, Lemma III.1.2]{brg}):
\begin{prop}\label{main}
    Let $A$ be a prime affine PI algebra over an algebraically closed field $\mathbb{k}$, with PI-deg($A$) = $n$ and $V$ be a simple $A$-module. Then $V$ is a vector space over $\mathbb{k}$ of dimension $t$, where $t \leq n$, and $A/ann_A(V) \cong M_t(\mathbb{k})$. Moreover, the upper bound $\pideg(A)$ is attained by some simple $A$-modules.
\end{prop}
\begin{rema}
In the roots of unity context, the algebra $\mathcal{H}_{p,q}$ is classified as a prime affine PI algebra. Consequently, according to Proposition \ref{main}, it is quite clear that each simple $\mathcal{H}_{p,q}$-module is finite-dimensional and can have dimension at most $\pideg(\mathcal{H}_{p,q})$. Hence the computation of the PI degree of $\mathcal{H}_{p,q}$ bears substantial importance. In Section \ref{pisection}, our focus will be on computing the PI degree of $\mathcal{H}_{p,q}$.
\end{rema}
\section{PI Degree for \texorpdfstring{$\mathcal{H}_{p,q}$}{TEXT}}\label{pisection}
In this section, we find an explicit expression of the PI degree for $\mathcal{H}_{p,q}$ assuming $p$ and $q$ are primitive $m$-th and $n$-th roots of unity respectively. Let $\Gamma$ be the multiplicative subgroup of $\mathbb{k}^*$ generated by $p$ and $q$. Clearly, $\Gamma$ is a cyclic group of order $l:=\lcmu(m,n)$. Here we will use the derivation erasing process \cite[Theorem 7]{lm2} and then a key technique for calculating the PI degree of a quantum affine space \cite[Proposition 7.1]{di}.\\
\noindent \textsc{Step 1:} Recall that the algebra $\mathcal{H}_{p,q}$ has a skew polynomial presentation of the form
\[\mathbb{k}[z][x,\alpha][y,\beta,\delta].\] Note that \[\delta(\beta(x))=qz=pqp^{-1}z=pq\beta(\delta(x)).\] This holds trivially if $x$ is replaced by $z$. So the pair $(\beta,\delta)$ is a $pq$-skew derivation on $\mathbb{k}[z][x,\alpha]$ with $pq\neq 1$. Moreover, we can check that all the hypothesis of the derivation erasing process in \cite[Theorem 7]{lm2} is satisfied by the skew polynomial presentation of the PI algebra $\mathcal{H}_{p,q}$. Hence it follows that 
\begin{equation}\label{equal}
\pideg (\mathcal{H}_{p,q})= \pideg \mathcal{O}_{\Lambda}(\mathbb{k}^{3})    
\end{equation}
where the $(3\times 3)$-matrix of relations $\Lambda$ is
\begin{equation*}\label{defma}
\Lambda=\begin{pmatrix}
1&p^{-1}&p\\
p&1&q^{-1}\\
p^{-1}&q&1
\end{pmatrix}.
\end{equation*}
\noindent\textsc{Step 2:} Suppose $g$ is a generator of the cyclic group $\Gamma$. Then we can choose $s_1=\frac{n}{\gcdi(m,n)}$ and $s_2=\frac{m}{\gcdi(m,n)}$ %$0<s_i\leq l$ with $s_i$ divides $l$, 
such that $\langle p\rangle=\langle g^{s_1}\rangle$ and $\langle q\rangle=\langle g^{s_2}\rangle$. Therefore there exist non-negative integers $k_1<m,k_2<n$ with $\gcdi(k_1,m)=\gcdi(k_2,n)=1$ such that 
\begin{equation}\label{cypq}
    p=g^{s_1k_1}\ \ \text{and}\ \ q=g^{s_2k_2}.
\end{equation} Since $pq\neq 1$, so we have $s_1k_1+s_2k_2\not\equiv 0~(\mo l)$ and hence $(k_1,k_2)\neq (0,0)$. Thus the integral matrix associated with $\Lambda$ is
\begin{equation*}\label{defma1}
H=\begin{pmatrix}
0&-s_1k_1&s_1k_1\\
s_1k_1&0&-s_2k_2\\
-s_1k_1&s_2k_2&0
\end{pmatrix}.
\end{equation*}
One of the key techniques for calculating the PI degree of a quantum affine space was first introduced in \cite[Proposition 7.1]{di} and later simplified with respect to invariant factors in \cite[Lemma 5.7]{ar}. As $H$ is a $(3\times 3)$-skew-symmetric matrix, so $\det(H)=0$ and the rank of $H$ is $2$. Suppose $h_1\divides h_1$ are invariant factors for $H$. Therefore we have \[h_1=\text{first~determinantal~divisor}=\gcdi(s_1k_1,s_2k_2)\]
\noindent\textsc{Step 3:} To obtain the $\pideg\mathcal{O}_{\Lambda}(\mathbb{k}^3)$, we now apply the result \cite[Lemma 5.7]{ar} in this context as below
\begin{equation}\label{pid7}
  \left(\pideg\mathcal{O}_{\Lambda}(\mathbb{k}^3)\right)^2=\frac{l}{\gcdi(h_1,l)}\times \frac{l}{\gcdi(h_1,l)}.  
\end{equation}
Finally, we claim that $\gcdi(h_1,l)=1$. Suppose $\gcdi(h_1,l)=d$. Then $d$ divides $s_1k_1,s_2k_2$ and $l$. Now we simplify the following
\[p^{\frac{l}{d}}=g^{\frac{s_1k_1l}{d}}=1\ \text{and}\ q^{\frac{l}{d}}=g^{\frac{s_2k_2l}{d}}=1.\] This implies that both $m$ and $n$ divide $l/d$ and hence $l$ divides $l/d$. Therefore the claim follows. Thus from (\ref{equal}) and (\ref{pid7}), we have now proved that
\begin{theo}
Suppose that $p$ and $q$ are primitive $m$-th and $n$-th roots of unity, respectively. Then the PI degree of $\mathcal{H}_{p,q}$ is given by $\pideg(\mathcal{H}_{p,q})=l$, where $l=\lcmu(m,n)$.
\end{theo}
\section{The Center of \texorpdfstring{$\mathcal{H}_{p,q}$}{TEXT}}\label{sec4}
In this section, we focus on computing the center of $\mathcal{H}_{p,q}$ in the roots of unity context. Let $p,q\in\mathbb{k}^*$ with $pq\neq 1$ be such that $\ord(p)=m$, and $\ord(q)=n$. Let $l$ denote the least common multiple of $m$ and $n$. Then it is easy to check using the defining relations and Lemma \ref{identity} that the elements $z^m,\theta^n,x^l$, and $y^l$ are in the center. Now we aim to establish an algebraic dependence relationship of $\theta^k$ for any $k\geq 1$. First, observe that $\theta=(q-p^{-1})xy+z$. Then we can simplify the expression $\theta^k$ using the following fact:
\begin{itemize}
    \item [(1)] The elements $z$ and $xy$ commute.
    \item[(2)] For any $r\geq 1$, $(x^ry^r)(xy)=q^rx^{r+1}y^{r+1}+p^r[r]_{p,q}zx^ry^r$.
\end{itemize}
Finally, we obtain the following identity, for any $k\geq 1$:
\begin{equation}\label{ident}
\theta^k= (q-p^{-1})^kq^{\frac{k(k-1)}{2}}x^ky^k+\sum\limits_{r=1}^{k-1}c^{(k)}_rx^ry^r+z^k 
\end{equation}
where each $c^{(k)}_r$ is an element in $\mathbb{k}[z]$. %Similarly, for the expression $pq\theta=(pq-1)yx+z$, we can obtain the following identity, for any $k\geq 1$:
%\begin{equation}\label{idnt}
%    (pq\theta)^k=(pq-1)^kq^{\frac{-k(k-1)}{2}}y^kx^k+\sum\limits_{r=1}^{k-1}d^{(k)}_ry^rx^r+z^k
%\end{equation} where each $d^{(k)}_r$ is an element in $\mathbb{k}[z]$.
Thus the identity (\ref{ident}) for $k=n$ provides a central element that in general not lie in the subalgebra generated by $z^m,x^l$, and $y^l$.
\par Denote the lexicographical order on $\mathbb{Z}^2_{+}$ by $(a',b')<(a,b)$, that means $a'<a$, or $a'=a$ and $b'<b$. For any $0\neq f\in \mathcal{H}_{p,q}$, we may uniquely write 
\[f=f_{uv}~x^uy^v+\sum\limits_{(a,b)<(u,v)}f_{ab}~x^ay^b\]
where each $f_{uv}\neq 0$ and $f_{ab}$ is an element in the subalgebra $\mathbb{k}[z]$. Denote the degree of $f$ by $\deg(f)=(u,v)$.
\par In the following, we compute the center of $\mathcal{H}_{p,q}$ when $\ord(p)=m$ and $\ord(q)=n$. Recall from Section \ref{pisection} that the cyclic group $\Gamma=\langle g\rangle$ generated by $p$ and $q$ is of order $l$ with $p=g^{s_1k_1}$ and $q=g^{s_2k_2}$ where $s_i,k_i$ are as defined in (\ref{cypq}). Clearly, the subgroup $\langle p\rangle \cap \langle q\rangle $ is cyclic with order $\gcdi(m,n)$. It follows from the equality $l=s_1s_2\gcdi(m,n)$ that the element $g^{s_1s_2}$ is a generator of this subgroup. Note that $\gcdi(m,n)$ is relatively prime to $\lcmu(k_1,k_2)$. This implies that $g^{s_1s_2\lcmu(k_1,k_2)}$ is also a generator of this subgroup. Thus we can write: 
\[g^{s_1s_2\lcmu(k_1,k_2)}=p^{\frac{s_2k_2}{\gcdi(k_1,k_2)}}=q^{\frac{s_1k_1}{\gcdi(k_1,k_2)}},\] where $k_1$ and $k_2$ are two parameters as in (\ref{cypq}) that parameterize the elements $p$ and $q$, respectively, in $\Gamma$. Take 
\begin{equation}\label{rs}
r:=\frac{s_2k_2}{\gcdi(k_1,k_2)}~(\mo m)\ \ ~~~~\text{and}~~~~\  s:=\frac{s_1k_1}{\gcdi(k_1,k_2)}~(\mo n)
\end{equation}
Thus we obtain a pair $(r,s)$ of non-negative integers satisfying $p^r=q^s$. Therefore all pairs $(a,b)$ of non-negative solutions for $p^a=q^b$ can be constructed from $(r,s)$ and are given by 
\begin{equation}\label{sol}
(a,b)=(rt+m\mathbb{Z}_{\geq 0},st+n\mathbb{Z}_{\geq 0}),\ t\in \mathbb{Z}_{\geq 0}.
\end{equation}
For such pair $(r,s)$, let us denote $\Omega:=z^r\theta^s$. Then we can check that the element $\Omega$ is in the center of $\mathcal{H}_{p,q}$. 
\begin{theo}
    The center of $\mathcal{H}_{p,q}$ is the $\mathbb{k}$-algebra generated by $z^m,\theta^n,x^l,y^l$, and $\Omega$. 
\end{theo}
\begin{proof}
Let $S$ denote the $\mathbb{k}$-algebra generated by $z^m,\theta^n,x^l,y^l$, and $\Omega$. It is evident that $S$ is a central subalgebra. Suppose 
\[f=f_{uv}~x^uy^v+\sum\limits_{(a,b)<(u,v)}f_{ab}~x^ay^b\] is in ${Z}(\mathcal{H}_{p,q})\setminus S$ with the minimal degree $(u,v)$. %From (\ref{ident}), it is important to note that $(u,v)<(n,n)$.  
Then we simplify the equality $xf=fx$ to obtain
    \[0=xf-fx=\left(xf_{uv}-q^vf_{uv}x\right)x^uy^v+\ \ \text{lower degree terms}.\]
This implies $xf_{uv}=q^vf_{uv}x$. From this, we can deduce that $q^v=p^k$ for all $k$ in the support of $f_{uv}$. Therefore from (\ref{sol}) we have $v=st+nt'$ and $k=rt+mt_k$ for some $t,t',t_k\in \mathbb{Z}_{\geq 0}$ and for all $k$ in the support of $f_{uv}$. Thus $f_{uv}$ is a product of $z^{rt}$ with a polynomial in $z^m$.
\par Similarly, from the equality $yf=fy$, we obtain $u=st+nt''$ for some $t''\in \mathbb{Z}_{\geq 0}$. Hence $n$ divides $(u-v)$. Then from $zf=fz$ we have $p^{u-v}=1$, which implies $m$ divides $(u-v)$. Now if $u\geq v$, then using (\ref{ident}), the \[f-(q-p^{-1})^{-v}q^{\frac{-v(v-1)}{2}}f_{uv}~x^{u-v}\theta^v\] is a nonzero element in ${Z}(\mathcal{H}_{p,q})\setminus S$, with degree less than $(u,v)$. On the other hand, if $v>u$, then using (\ref{ident}), the \[f-(q-p^{-1})^{-u}q^{\frac{-u(u-1)}{2}}f_{uv}~\theta^uy^{v-u}\] is a nonzero element in ${Z}(\mathcal{H}_{p,q})\setminus S$, with degree less than $(u,v)$. Thus either case contradicts the fact that the degree $(u,v)$ of $f$ is minimal. This completes the proof.  
\end{proof}

\begin{coro}
(1) When $\gcdi(m,n)=1$, the pair $(r,s)$ in (\ref{rs}) will be equal to $(0,0)$ and hence $\Omega=1$. Therefore the center of $\mathcal{H}_{p,q}$ is the $\mathbb{k}$-algebra generated by $z^m,\theta^n,x^{mn}$ and $y^{mn}$.
\par (2) When $p=1$, the center of $\mathcal{H}_{1,q}$ is generated by the $z,\theta^n,x^n$, and $y^n$. Moreover rearranging the identity (\ref{ident}) for $k=n$ we have
 \[\theta^n-(q-1)^nq^{\frac{n(n-1)}{2}}x^ny^n-z^n=\sum\limits_{r=1}^{n-1}c^{(n)}_rx^ry^r.\] As the left-hand side is central, we can conclude that $c^{(n)}_r=0$ for $1\leq r\leq n-1$. Therefore we obtain 
 \begin{equation}\label{tiden}
     \theta^n=(q-1)^nq^{\frac{n(n-1)}{2}}x^ny^n+z^n.
 \end{equation} Thus the center of $\mathcal{H}_{1,q}$ is the polynomial algebra generated by $z,x^n$, and $y^n$.
 \par (3) When $q=1$, the center of $\mathcal{H}_{p,1}$ is generated by $z^m,\theta,x^m$ and $y^m$. We now define a $\mathbb{k}$-linear map $\psi:\mathcal{H}_{p,1}\rightarrow \mathcal{H}_{1,p^{-1}}$ by
\[\psi(x)=x,\ \psi(y)=y,\ \psi(z)=\theta,\ \psi(\theta)=z.\] Then we can easily verify that $\psi$ is an isomorphism. Using (\ref{tiden}) in $\mathcal{H}_{1,p^{-1}}$ and the above isomorphism, we obtain
\[z^m=(p^{-1}-1)^mp^{\frac{-m(m-1)}{2}}x^my^m+\theta^m \ \ \text{in}\ \ \mathcal{H}_{p,1}.\] Thus the center of $\mathcal{H}_{p,1}$ is the polynomial algebra generated by $\theta,x^m$, and $y^m$.
\end{coro}
\section{Construction of Simple Modules over \texorpdfstring{$\mathcal{H}_{p,q}$}{TEXT}}\label{sec5}
In this section, we construct simple modules over $ \mathcal{H}_{p,q} $ depending on some scalar parameters. Let $p$ and $q$ be primitive $m$-th and $n$-th roots of unity, respectively, and let $l$ denote the least common multiple of $m$ and $n$.
\subsection{Simple Modules of Type \texorpdfstring{$\mathcal{V}_1(\mu,\lambda,\gamma)$}{TEXT}} Given $(\mu,\lambda,\gamma)\in (\mathbb{k}^*)^3$, let $\mathcal{V}_1(\mu,\lambda,\gamma)$ denote the $\mathbb{k}$-vector space with basis $\{v_k~|~0\leq k\leq l-1\}$. We define the $\mathcal{H}_{p,q}$-module structure on $\mathcal{V}_1(\mu,\lambda,\gamma)$ with action given by
\begin{align}
\ v_kz&=p^k\lambda v_k,\ \ v_kx=\begin{cases}
    \mu v_{k+1},& 0\leq k\leq l-2\\
    \mu v_0,&k=l-1
\end{cases} \label{1}\\
\ v_ky&=\begin{cases}
    \mu^{-1}q^{-k}\displaystyle\frac{pq\gamma-(pq)^k\lambda}{pq-1}v_{k-1},& 1\leq k\leq l-1\\
    \mu^{-1}\displaystyle\frac{pq\gamma-\lambda}{pq-1}v_{l-1},&k=0
\end{cases}\label{2}
\end{align}
We shall see that (\ref{1}) and (\ref{2}) does indeed define an $\mathcal{H}_{p,q}$-module. For $0\leq k\leq l-1$,
\[v_k(zx-p^{-1}xz)=p^k\lambda v_kx-p^{-1}\mu v_{k+1}z=p^k\lambda \mu v_{k+1}-p^{-1}\mu p^{k+1}\lambda v_{k+1}=0,\] Similarly, $v_k(zy-pyz)=0$. Finally,
\begin{align*}
    v_k(yx-qxy-z)&=q^{-k}\displaystyle\frac{pq\gamma-(pq)^k\lambda}{pq-1}v_k-q^{-k}\displaystyle\frac{pq\gamma-(pq)^{k+1}\lambda}{pq-1}v_k-p^k\lambda v_k\\
    &=q^{-k}\displaystyle\frac{(pq)^{k+1}\lambda-(pq)^k\lambda}{pq-1}v_k-p^k\lambda v_k=0.
\end{align*}
Thus $\mathcal{V}_1(\mu,\lambda,\gamma)$ is a module over $\mathcal{H}_{p,q}$. It is important to note that the action of the normal element $\theta=yx-p^{-1}xy$ on $\mathcal{V}_1(\mu,\lambda,\gamma)$ is $v_k\theta=q^{-k}\gamma v_k$ for all $0\leq k\leq l-1$.
\begin{theo}\label{t1}
    The $\mathcal{H}_{p,q}$-module $\mathcal{V}_1(\mu,\lambda,\gamma)$ is a simple module of dimension $l$.
\end{theo}
\begin{proof}
Suppose that $W$ is a nonzero submodule of $\mathcal{V}_1(\mu,\lambda,\gamma)$. We claim that $W=\mathcal{V}_1(\mu,\lambda,\gamma)$. Let $w=\sum\limits_{i=0}^{l-1}\xi_iv_i\in W$, with $\xi_i\in \mathbb{k}$, be a nonzero element of $W$. Suppose there are two nonzero scalars, say $\xi_k,\xi_s$ in the expression of $w$. Since $W$ is a submodule then the vectors $wz$ and $w\theta$ are in $W$, where
\[wz=\sum\limits_{i=0}^{l-1}\xi_ip^{i}\lambda v_i\ \ \text{and}\ \ w\theta=\sum\limits_{i=0}^{l-1}\xi_iq^{-i}\gamma v_i.\]
\noindent\textsc{Case 1:} Assume that $k\not\equiv s~(\mo~m)$. Then we consider $wz-p^k\lambda w\in W$, where
\[wz-p^{k}\lambda w=\sum\limits_{i=0,i\neq k}^{l-1}(p^{i}-p^{k})\lambda\xi_iv_i.\] This is nonzero since $m$ does not divide $(k-s)$ and $\lambda\xi_s\neq 0$.\\
\noindent\textsc{Case 2:} Assume that $k\equiv s~(\mo m)$. Then we consider $w\theta-q^{-k}\gamma w\in W$, where
\[w\theta-q^{-k}\gamma w=\sum\limits_{i=0,i\neq k}^{l-1}(q^{-i}-q^{-k})\gamma\xi_iv_i.\] If $w\theta-q^{-k}\gamma w=0$, then we obtain $q^{-s}=q^{-k}$. This implies $n$ divides $(k-s)$. Now by our assumption, $l$ divides $(k-s)$. This contradicts the fact that $k$ and $s$ are distinct indices. Therefore $w\theta-q^{-k}\gamma w$ is nonzero.
\par Thus in either case $(wz-p^{k}\lambda w)$ or $(w\theta-q^{-k}\gamma w)$ is a nonzero element in $W$ of smaller length than $w$. Hence by induction, there must exist some $0\leq i\leq l-1$ such that $v_i\in W$. Since the action of $x$ permutes all the basis vectors, we have $W=\mathcal{V}_1(\mu,\lambda,\gamma)$.
\end{proof}
\subsection{Simple Modules of Type \texorpdfstring{$\mathcal{V}_2(\mu,\lambda)$}{TEXT}} Given $(\mu,\lambda)\in (\mathbb{k}^*)^2$, let $\mathcal{V}_2(\mu,\lambda)$ denote the $\mathbb{k}$-vector space with basis $\{v_k~|~0\leq k\leq l-1\}$. We define the action of the generators of $\mathcal{H}_{p,q}$ on $\mathcal{V}_2(\mu,\lambda)$ as follows
\begin{align}
v_kz&=p^{-k}\lambda v_k,\ \ v_ky=\begin{cases}
    \mu v_{k+1},& 0\leq k\leq l-2\\
    \mu v_0,&k=l-1
\end{cases} \label{3}\\
v_kx&=\begin{cases}
    \mu^{-1}\lambda[k]_{p,q} v_{k-1},& 1\leq k\leq l-1\\
    0,&k=0
\end{cases}\label{4}
\end{align}
It is easy to verify that the above action establishes a $\mathcal{H}_{p,q}$-module structure on $\mathcal{V}_2(\mu,\lambda)$. The action of the normal element $\theta=yx-p^{-1}xy$ on $\mathcal{V}_2(\mu,\lambda)$ is given by $v_k\theta=\lambda q^kv_k$ for all $0\leq k\leq l-1$.
\begin{theo}
The $\mathcal{H}_{p,q}$-module $\mathcal{V}_2(\mu,\lambda)$ is a simple module of dimension $l$.
\end{theo}
\begin{proof}
    The proof follows a similar line of reasoning as presented in Theorem \ref{t1}.
\end{proof}
\subsection{Simple Modules of Type \texorpdfstring{$\mathcal{V}_3(\lambda)$}{TEXT}} Given $\lambda\in \mathbb{k}^*$, let $\mathcal{V}_3(\lambda)$ denote the $\mathbb{k}$-vector space with basis $\{v_k~|~0\leq k\leq \ord(pq)-1\}$. Let us define the $\mathcal{H}_{p,q}$-module on $\mathcal{V}_3(\lambda)$ with action given by
\begin{align}
v_kz&=p^{-k}\lambda v_k,\ \ v_ky=\begin{cases}
    v_{k+1},& 0\leq k\leq \ord(pq)-2\\
    0,&k=\ord(pq)-1
\end{cases} \label{5}\\
v_kx&=\begin{cases}
    \lambda[k]_{p,q} v_{k-1},& 1\leq k\leq \ord(pq)-1\\
    0,&k=0
\end{cases}\label{6}
\end{align}
This action indeed defines a $\mathcal{H}_{p,q}$-module structure on $\mathcal{V}_3(\lambda)$. Moreover the action of the normal element $\theta=yx-p^{-1}xy$ on $\mathcal{V}_3(\lambda)$ is given by $v_k\theta=q^k\lambda v_k$ for all $0\leq k\leq \ord(pq)-1$. 
\begin{theo}
    The $\mathcal{H}_{p,q}$-module $\mathcal{V}_3(\lambda)$ is a simple module of dimension $\ord(pq)$.
\end{theo}
\begin{proof}
    The proof follows a parallel line of reasoning as in Theorem \ref{t1}.
\end{proof}
In Section \ref{maximum}, we intend to identify the conditions on $m$ and $n$ for which the $\ord(pq)$ is either at its maximum or non-maximum values in the cyclic group $\Gamma$.
\section{Classification of Simple Modules over \texorpdfstring{$\mathcal{H}_{p,q}$}{TEXT}}\label{sec6}
In this section, we completely classify all simple $\mathcal{H}_{p,q}$-modules, assuming $p$ and $q$ are primitive $m$-th and $n$-th roots of unity respectively and $\mathbb{k}$ is an algebraically closed field. Let $\mathcal{N}$ be a simple module over $\mathcal{H}_{p,q}$. Then by Proposition \ref{main}, the $\mathbb{K}$-dimension of $\mathcal{N}$ is finite and bounded above by $\pideg(\mathcal{H}_{p,q})$. This classification of simple modules is based on the action of appropriate central or normal elements within the algebra $\mathcal{H}_{p,q}$. It is important to note that the elements $z$ and $\theta$ %commute with each other in $\mathcal{H}_{p,q}$. Therefore there is a common eigenvector $v$ in $\mathcal{N}$ of the commuting operators $z$ and $\theta$. Take
%\[vz=\lambda_1v,\ v\theta=\lambda_2v,\ \ \lambda_1,\lambda_2\in \mathbb{K}.\]
%We have noted that both $z$ and $\theta$ 
are normal elements. In view of Lemma \ref{itn}, the action of each normal element (namely, $z$ and $\theta$) on a simple module $\mathcal{N}$ is either trivial or invertible. Based on this fact, we will now consider the following subsections:
\subsection{Simple \texorpdfstring{$z$}{TEXT}-torsion \texorpdfstring{$\mathcal{H}_{p,q}$}{TEXT}-modules}\label{ztorsion} In this case, the action of $z$ on $ \mathcal{N}$ is trivial. Then $\mathcal{N}$ becomes a simple module over the factor algebra $\mathcal{H}_{p,q}/\langle{z}\rangle$ which is isomorphic to a quantum plane 
\[\mathcal{O}_{{q}}(\mathbb{k}^2)=\langle y,x~|~yx=qxy\rangle.\]
It is important to note that the quantum plane $\mathcal{O}_{{q}}(\mathbb{k}^2)$ serves as a prime affine PI algebra and its PI degree is given by $n$ (cf. \cite{di}). The simple modules over $\mathcal{O}_{{q}}(\mathbb{k}^2)$ have already been classified in \cite{moh}. Here the possible $\mathbb{K}$-dimensions of $\mathcal{N}$ are $1$ (if $\mathcal{N}$ is $x$-torsion or $y$-torsion) or $n$ (if $\mathcal{N}$ is $x,y$-torsionfree).
\subsection{Simple \texorpdfstring{$\theta$}{TEXT}-torsion \texorpdfstring{$\mathcal{H}_{p,q}$}{TEXT}-modules}\label{thetatorsion} In this case, the action of $\theta$ on $ \mathcal{N}$ is trivial. Then $\mathcal{N}$ becomes a simple module over the factor algebra $\mathcal{H}_{p,q}/\langle{\theta}\rangle$ which is isomorphic to a quantum plane \[\mathcal{O}_{{p}}(\mathbb{k}^2)=\langle x,y~|~xy=pyx\rangle.\]
In the root of unity context, the quantum plane $\mathcal{O}_{{p}}(\mathbb{k}^2)$ represents a prime affine PI algebra, and its PI degree is $m$ (cf. \cite{di}). The simple modules over $\mathcal{O}_{{p}}(\mathbb{k}^2)$ have already been classified in \cite{moh}. Here the possible $\mathbb{K}$-dimensions of $\mathcal{N}$ are $1$ (if $\mathcal{N}$ is $x$-torsion or $y$-torsion) or $m$ (if $\mathcal{N}$ is $x,y$-torsionfree).
\subsection{Simple \texorpdfstring{$z,\theta$}{TEXT}-torsionfree \texorpdfstring{$\mathcal{H}_{p,q}$}{TEXT}-modules}\label{torfree} In that case, the action of both $z$ and $\theta$ on $\mathcal{N}$ are invertible. Let $l$ denote the least common multiple of $m$ and $n$. It is important to note that the elements $ x^l,y^l,z$, and $\theta$ commute with each other in $\mathcal{H}_{p,q}$. Since $\mathcal{N}$ is finite-dimensional, there is a common eigenvector $v$ in $\mathcal{N}$ of these commuting operators. So we can take
\[vx^l = \alpha v,~~ v y^{l} = \beta v,~~vz=\lambda_1v,\ v\theta=\lambda_2v\]
for some $\alpha,\beta \in \mathbb{k}$ and $\lambda_1,\lambda_2\in\mathbb{k}^*$. Note that the $ x^l$ and $y^l$ are central elements of $\mathcal{H}_{p,q}$, so by Schur's lemma, they act as multiplication by scalar on $\mathcal{N}$. %Henceforth we can assume \[vx^l = \alpha v,~~ v y^{l} = \beta v,\] for some $ \alpha,\beta \in \mathbb{K}$ where $v$ is common eigenvector of $yx$ and $ z$. 
In the following we shall determine the structure of simple $ \mathcal{H}_{p,q}$-module $ \mathcal{N}$ according to the scalar parameters $\alpha$ and $\beta$: \\
\textsc{ Case I:} Let us first assume that $\alpha \neq 0$. Then $x$ acts as an invertible operator on $\mathcal{N}$. Therefore the vectors $vx^k$ where $ 0\leq k \leq l-1 $ of $\mathcal{N}$ are nonzero. Let us set $(\mu,\lambda,\gamma):=(\alpha^{\frac{1}{l}},\lambda_1,\lambda_2)\in(\mathbb{k}^*)^3$. Define a $\mathbb{k}$-linear map 
$\Phi:\mathcal{V}_1(\mu,\lambda,\gamma)\rightarrow \mathcal{N}$ by $\Phi(v_k)=\mu^{-k}vx^k$ for all $0\leq k\leq l-1$. We can easily verify that $\Phi$ is a nonzero $\mathcal{H}_{p,q}$-module homomorphism. Thus by Schur's lemma, $\Phi$ is an isomorphism.\\
\noindent\textsc{Case II:} Assume that $\alpha =0$ and $\beta\neq 0$. Then the operators $x$ and $y$ on $\mathcal{N}$ are nilpotent and invertible respectively. Now consider the $\mathbb{k}$-space $\ker(x):=\{w\in\mathcal{N}~|~wx=0\}$. Clearly $\ker(x)\neq \varphi$. From the defining relations, we can verify that each of the operators $y^l$ and $z$ keeps the $\mathbb{k}$-space $\ker(x)$ invariant. Therefore there is a common eigenvector $w\in \ker(x)$ of the commuting operators $y^l$ and $z$. Let us take
\[wx=0,\ wy^l=\beta'w,\ wz=\lambda'w.\] By our hypothesis, we have $\beta'$ and $\lambda'$ both are nonzero. Set $(\mu,\lambda):=((\beta')^{\frac{1}{l}},\lambda')\in (\mathbb{k}^*)^2$. Now define a $\mathbb{k}$-linear map $\Phi:\mathcal{V}_2(\mu,\lambda)\rightarrow \mathcal{N}$ by $\Phi(v_k)=\mu^{-k}wy^k$ for all $0\leq k\leq l-1$. It is easy to check that $\Phi$ is a nonzero $\mathcal{H}_{p,q}$-module homomorphism. Thus by Schur's lemma, $\Phi$ is a module isomorphism.\\
\noindent\textsc{Case III:} Assume that $ \alpha=\beta=0$. Then both $x$ and $y$ are nilpotent operators on $\mathcal{N}$. Similar to Case II, there is a common eigenvector $w\in \ker(x)$ of the commuting operators $y^l$ and $z$. As $y$ is a nilpotent operator, take
\[wx=0,\ wy^l=0,\ wz=\lambda w,\ \ \lambda \neq 0.\]
Let $r$ be the smallest index such that $wy^{r-1} \neq 0 $ and $ wy^{r}=0$, where $1\leq r\leq l$. We now claim that $r=\ord(pq)$. Indeed, by simplifying the equality $wy^rx=0$, we have 
\[0=wy^rx=q^rwxy^r+[r]_{p,q}wzy^{r-1}=[r]_{p,q}\lambda wy^{r-1}.\]
This implies $[r]_{p,q}=0$. As $r$ is the smallest nonzero index, we can obtain $r=\ord(pq)$.
\par Thus the vectors $wy^k$ for $0\leq k\leq \ord(pq)-1$ are nonzero. Now define a $\mathbb{k}$-linear map $\Phi:\mathcal{V}_3(\lambda)\rightarrow \mathcal{N}$ by setting $\Phi(v_k)=wy^k$ for all $0\leq k\leq \ord(pq)-1$. This is a nonzero $\mathcal{H}_{p,q}$-module homomorphism. Thus by Schur's lemma, $\Phi$ is a module isomorphism.
\par Finally, the above discussion provides an opportunity for the classification of simple $\mathcal{H}_{p,q}$-modules in terms of scalar parameters. 
%Based on the discussion above, this section yields the following main results.
\begin{theo}\label{classi}
    Suppose $\mathcal{N}$ be a simple $z,\theta$-torsionfree $\mathcal{H}_{p,q}$-module. Then 
    \begin{itemize}
        \item[(1)] $\mathcal{N}$ is isomorphic to $\mathcal{V}_1(\mu,\lambda,\gamma)$ if $\mathcal{N}$ is $x$-torsionfree.
        \item [(2)] $\mathcal{N}$ is isomorphic to $\mathcal{V}_2(\mu,\lambda)$ if $\mathcal{N}$ is $x$-torsion and $y$-torsionfree.   
        \item[(3)] $\mathcal{N}$ is isomorphic to $\mathcal{V}_3(\lambda)$ if $\mathcal{N}$ is $x,y$-torsion.
    \end{itemize} 
\end{theo}
Based on the classification, we can derive the following:
\begin{theo}\label{nesu}
Suppose $\ord(p)=m$ and $\ord(q)=n$ with $m\notdivides n$ and $n \notdivides m$ satisfying $\ord(pq)=\lcmu(m,n)$. The existence of such a permissible pair $(m,n)$ is stated in Proposition \ref{suff}. Then the following are equivalent
\begin{itemize}
    \item [(1)] $\mathcal{N}$ is maximal-dimensional simple $\mathcal{H}_{p,q}$-module, i.e., $\dime_{\mathbb{k}}(\mathcal{N})=\pideg \mathcal{H}_{p,q}$.
    \item [(2)] $\mathcal{N}$ is $z,\theta$-torsionfree simple $\mathcal{H}_{p,q}$-module.
\end{itemize}
%Then a simple $\mathcal{H}_{p,q}$-module $\mathcal{N}$ is maximal-dimensional, i.e., $\dime_{\mathbb{k}}(\mathcal{N})=\pideg \mathcal{H}_{p,q}$ if and only if $\mathcal{N}$ is $z,\theta$-torsionfree $\mathcal{H}_{p,q}$-module.
\end{theo}
In general, the non-divisibility hypothesis stated in Theorem \ref{nesu} should not be omitted, as mentioned below:
\begin{coro} Suppose $\ord(p)=m$ and $\ord(q)=n$. Then
\begin{itemize}
    \item [(1)] If $m\divides n$, then there is a $z$-torsion and $x,y$-torsionfree simple $\mathcal{H}_{p,q}$-modules of maximal-dimensional. (sec Subsection \ref{ztorsion}).
    \item [(2)] If $n\divides m$, then there is a $\theta$-torsion and $x,y$-torsionfree simple $\mathcal{H}_{p,q}$-modules of maximal-dimensional. (sec Subsection \ref{thetatorsion}).
    \end{itemize}
    \end{coro}
%If $m\notdivides n$ and $n \notdivides m$, then a simple $\mathcal{H}_{p,q}$-module $\mathcal{N}$ is maximal-dimensional, i.e., $\dime_{\mathbb{k}}(\mathcal{N})=\pideg \mathcal{H}_{p,q}$ if and only if $\mathcal{N}$ is $z,\theta$-torsionfree module.
\section{Isomorphisms between simple \texorpdfstring{$\mathcal{H}_{p,q}$}{TEXT}-modules}\label{sec7}
In this section, we study when two of the modules appearing in the classification given in Theorem \ref{classi} are isomorphic. %It follows from the classification that one module from one of the types in that theorem cannot be isomorphic to one from any of the other types, but it is possible for two distinct ones of the same type to be isomorphic to each other. 
Consequent to the classification, it can be deduced that a module belonging to any of the types described in that theorem cannot exhibit isomorphism with a module from any other type. However, it is possible for two distinct ones of the same type to be isomorphic to each other. Suppose $p,q\in \mathbb{k}^*$ with $pq\neq 1$ is such that $\ord(p)=m,\ord(q)=n$ and $l=\lcmu(m,n)$. 
\begin{prop}\label{1iso}
  Let $(\mu,\lambda,\gamma)$ and $(\mu',\lambda',\gamma')$ belong to $(\mathbb{k}^*)^3$. Then $\mathcal{V}_1(\mu,\lambda,\gamma)\cong \mathcal{V}_1(\mu',\lambda',\gamma')$ if and only if $\mu^l=(\mu')^l$ and there exists $0\leq k\leq l-1$ such that $\lambda=p^k\lambda'$ and $\gamma=q^{-k}\gamma'$.
\end{prop}
\begin{proof}
Let $\phi:\mathcal{V}_1(\mu,\lambda,\gamma)\rightarrow \mathcal{V}_1(\mu',\lambda',\gamma')$ be a module isomorphism. As $v_i=\mu^{-i}v_0x^i$ holds in $\mathcal{V}_1(\mu,\lambda,\gamma)$, therefore $\phi$ can be uniquely determined by $\phi(v_0)$. Suppose that 
\[\phi(v_0)=\sum\limits_{i=0}^{l-1}\xi_iv_i\] for at least one of $\xi_i\in \mathbb{k}^*$. If there are two nonzero coefficients, say $\xi_k$ and $\xi_s$, then the equalities 
\[\phi(v_0z)=\phi(v_0)z\ \ \text{and}\ \ \phi(v_0\theta)=\phi(v_0)\theta\] imply that 
\[\lambda=\lambda'p^k=\lambda'p^s\ \ \text{and}\ \ \gamma=\gamma'q^{-k}=\gamma'q^{-s}.\] This implies $(k-s)$ is a common divisor of $\ord(p)$ and $\ord(q)$, which is a contradiction. Therefore $\phi(v_0)=\xi v_k$ for some $\xi\in \mathbb{k}^*$ and for some $k$ with $0\leq k\leq l-1$. With this form of $\phi$, we can easily obtain the required relations.
\par Conversely, assume the relations between $(\mu,\lambda,\gamma)$ and $(\mu',\lambda',\gamma')$. Then define a $\mathbb{k}$-linear map $\psi:\mathcal{V}_1(\mu,\lambda,\gamma)\rightarrow\mathcal{V}_1(\mu',\lambda',\gamma')$ by $\psi(v_i)=(\mu^{-1}\mu')^iv_{k\oplus i}$, where $\oplus$ is the addition modulo $l$. It is easy to verify that $\psi$ is a module isomorphism.
\end{proof}
Next, we see when two distinct simple modules of type $\mathcal{V}_2(\mu,\lambda)$ are isomorphic. 
\begin{prop}\label{2iso}
  Let $(\mu,\lambda)$ and $(\mu',\lambda')$ belong to $(\mathbb{k}^*)^2$. Then $\mathcal{V}_2(\mu,\lambda)\cong \mathcal{V}_2(\mu',\lambda')$ if and only if $\mu^l=(\mu')^l$ and $\lambda=\lambda'$.
\end{prop}
\begin{proof}
    Let $\phi:\mathcal{V}_1(\mu,\lambda)\rightarrow \mathcal{V}_1(\mu',\lambda')$ be a module isomorphism. Using a similar argument as in Proposition \ref{1iso}, we have $\phi(v_0)=\xi v_k$ for some $\xi\in \mathbb{k}^*$ and for some $k$ with $0\leq k\leq l-1$. Simplify the equality $\phi(v_0x)=\phi(v_0)x$ to obtain $0=\xi v_kx$ in $\mathcal{V}_1(\mu',\lambda')$. Therefore it follows that $k=0$. Thus $\phi(v_0)=\xi v_0$ for some $\xi\in \mathbb{K}^*$ and the required relations can be obtained from this. The converse part follows from the converse part of Proposition \ref{1iso} with $k=0$.
\end{proof}
Finally, we see that two distinct simple modules of type $\mathcal{V}_3(\lambda)$ are not isomorphic. 
\begin{prop}
  Let $\lambda,\lambda'\in \mathbb{k}^*$. Then $\mathcal{V}_3(\lambda)\cong \mathcal{V}_3(\lambda')$ if and only if $\lambda=\lambda'$.
\end{prop}
\begin{proof}
    It follows from Proposition \ref{2iso}, given the module structure on $\mathcal{V}_3(\lambda)$.
\end{proof}

\section{When \texorpdfstring{$\ord(pq)$}{TEXT} is maximum?}\label{maximum}
Let $p,q\in\mathbb{k}^*$ with $pq\neq 1$ be such that $\ord(p)=m$ and $\ord(q)=n$. In this section, we will derive an expression for $\ord(pq)$ in the cyclic group $\Gamma$ generated by $p$ and $q$. Additionally, we will classify the pairs $(m, n)$ that result in either the maximum or non-maximum values of $\ord(pq)$ in $\Gamma$. Let $l$ denote the least common multiple of $m$ and $n$. Clearly, the order of $\Gamma$ is $l$. From Step 2 of Section \ref{pisection}, we recall the expression of $p$ and $q$ in terms of the generator $g$ of $\Gamma$, i.e., $\langle p\rangle=\langle g^{s_1}\rangle$ and $\langle q\rangle=\langle g^{s_2}\rangle$ where $s_1=\frac{n}{\gcd(m,n)}$ and $s_2=\frac{m}{\gcdi(m,n)}$. Then $l=s_1m=s_2n$. Therefore we can write 
\begin{equation}\label{order}
p=g^{s_1k_1}\ \ \text{and}\ \ q=g^{s_2k_2}    
\end{equation}
such that $\gcdi(k_1,m)=\gcdi(k_2,n)=1$. Now the element $pq$ in the cyclic group $\Gamma$ is of the form $pq=g^{s_1k_1+s_2k_2}$ with $s_1k_1+s_2k_2=\frac{nk_1+mk_2}{\gcdi(m,n)}$.
%As $pq\neq 1$, so we get either $n\neq 1$ or $m\neq 1$.
Then we can compute 
\begin{equation}\label{ordpq}
    \ord(pq)=\displaystyle\frac{l}{\gcdi\left({s_1k_1+s_2k_2},l\right)}=\displaystyle\frac{mn}{\gcdi(mk_2+nk_1, mn)}.
\end{equation} 
Thus we have now established that 
\begin{prop}
The order of the element $pq$ in $\Gamma$ is equal to the order of $mk_2+nk_1$ in the group $\mathbb{Z}_{mn}$, where $k_1$ and $k_2$ are two parameters as in (\ref{order}) that parameterize the elements $p$ and $q$, respectively, in $\Gamma$.
\end{prop} 
%Therefore the $\ord(pq)=\lcmu(m,n)$ if and only if $\gcdi(m+n,mn)=\gcdi(m,n)$. This will happen if $\gcdi(m,n)=1$ or $\gcdi(m+n,mn)$ is a product of distinct primes.
\begin{rema}\label{eqdiv}
It is important to note that the $\ord(pq)$ is a divisor of $l=\lcmu(m,n)$. Moreover by (\ref{ordpq}), the $\ord(pq)=l$ if and only if $\gcd(mk_2+nk_1,mn)=\gcd(m,n)$. Clearly, the $\ord(pq)=l$ if $\gcdi(m,n)=1$ or $m=n$ is an odd prime. Now we can observe the following:
   \begin{itemize}
          \item[(1)] The $\gcdi(m,n)$ is a divisor of $\gcdi(mk_2+nk_1,mn)$.
          \item [(2)] Every prime divisor of $\gcdi(mk_2+nk_1,mn)$ is also a prime divisor of $\gcdi(m,n)$.
      \end{itemize} 
Thus the set of prime divisors for both $\gcdi(mk_2+nk_1, mn)$ and $\gcdi(m,n)$ are identical.
\end{rema}
For any positive integer $k$, let $e_{k}(p_i)$ denote the exponent of the prime $p_i$ in the prime factorization of $k$. Clearly $e_k(p_i)=0$ if $p_i$ is not a prime factor of $k$. In the following, we will establish the necessary and sufficient conditions for $\ord(pq)$ to always be maximum in $\Gamma$.
\begin{prop}\label{suff}
    The $\ord(pq)=l$ for all $p,q\in \mathbb{k}^*$ with $\ord(p)=m$ and $\ord(q)=n$ if and only if either $m=n$ is an odd prime or $|e_m(p_i)-e_n(p_i)|\geq 1$ for all primes $p_i$ in the prime factorization of $m$ or $n$. 
    \end{prop}
    \begin{proof}
    Suppose that the $\ord(pq)=l$ for all $p,q\in \mathbb{k}^*$ with $\ord(p)=m$ and $\ord(q)=n$. If $|e_m(p_i)-e_n(p_i)|\geq 1$ for all primes $p_i$ in the prime factorization of $m$ or $n$, then we are done.
    %\noindent\textsc{Case 1:} Suppose $p_1$ is the only prime divisor of $m$ and $n$ and $e_m(p_1)=e_n(p_1)$. We now argue that $m=n$ is an odd prime. Indeed if $e_m(p_1)=e_n(p_1)>1$, then we can choose $s_1=s_2=k_1=1$ and $k_2=p_1-1$ such that $pq=g^{p_1}$ which has order less than $l$. Thus $m=n=p_1$. Moreover, by our hypothesis, $p_1$ must be odd.\\ \noindent\textsc{Case 2:} Suppose that there are two distinct prime divisors for $m$ or $n$. 
    If not let $e_m(p_i)=e_n(p_i)$ for some prime divisor $p_i$ of $m$ and $n$. We first argue that $m=n$. Then we set \[s_1=\frac{n}{\gcdi(m,n)},\ k_1=1,\ s_2=\frac{m}{\gcdi(m,n)}.\] As $\gcdi(p_i,s_2)=1$, there exists $a,b\in \mathbb{Z}$ such that 
    \begin{equation}\label{prord}
    p_ia-s_2b=1    
    \end{equation}
    %Now multiplying both sides by $s_1$ we have $p_i(as_1)-s_2(bs_1)=s_1$. 
    Note that $\gcdi(p_i,bs_1)=1$. Thus by Dirichlet prime number theorem, there exists a positive integer $t$ such that $bs_1+p_it$ is relatively prime to $n$. Set $ k_2 := bs_1+p_it~ ( \mo n )$, consequently $\gcdi(k_2,n)=1$. %Therefore we obtain a $k_2$ such that $\gcdi(k_2,n)=1$ and 
    Now multiplying (\ref{prord}) by $s_1$ and using the expression of $k_2$ we obtain $p_i$ divides $s_1k_1+s_2k_2$. For this choice, take $p=g^{s_1k_1}$ and $q=g^{s_2k_2}$ so that $\ord(p)=m$ and $\ord(q)=n$. Now we can observe that $\gcdi(s_1k_1,s_2k_2)=1$ and hence if $s_1k_1+s_2k_2\equiv 0 ~(\mo l)$, then $s_1=s_2=1$, i.e., $m=n$. Finally, we can easily verify from (\ref{ordpq}) that $1<\ord(pq)<l$ except for the possibility $m=n$. Thus we obtain $m=n$. 
    %\par Let $\mathcal{P}$ be the set of all prime divisors of $m(=n)$. By our hypothesis, we have $m=n\neq 2$. We now argue that $m=n$ is an odd prime. Suppose $\sum\limits_{p_i\in \mathcal{P}}e_{m}(p_i)\geq 2$. 
    
    We now argue that $m=n$ is an odd prime. Let $p_j$ denote the least prime divisor of $m$. Then we set $s_1=k_1=s_2=1$ and $k_2=p_j-1$. For this choice, we can easily verify from (\ref{ordpq}) that $1<\ord(pq)<l$ except for the possibility $m=n=p_j$. Moreover, by our hypothesis, $p_j$ must be odd. This completes the first part proof.
      %Suppose $m=p_1^{a_1}\cdots p_r^{a_r}$ and $n=p_1^{b_1}\cdots p_r^{b_r}$ are prime factorization of $m$ and $n$ respectively. Then $|a_i-b_i|\geq 1$ for all $i$. 
      %By (\ref{ordpq}) it is enough to show that $\gcdi(mk_2+nk_1,mn)=\gcdi(m,n)$. 
      %\par Let $p_i$ be a prime factor of $m$ or $n$ such that $e_m(p_i)=a$ and $e_n(p_i)=b$. Again by the hypothesis means $|a-b|\geq 1$. Without loss of generality, we may assume that $a>b$. Therefore $e_{\gcdi(m,n)}(p_i)=b$. Then by observation (1), we can say $e_{\gcdi(mk_2+nk_1, mn)}(p_i)=c$ for $c\geq b$. Therefore $e_{mk_2+nk_1}(p_i)\geq c$ and $e_{mn}(p_i)=a+b\geq c$. If $a\geq c$, then $e_m(p_i)\geq c$ and hence $e_{n}(p_i)\geq c$. This implies $e_{\gcdi(m,n)}(p_i)\geq c$ and hence $b=c$. On the other hand, if $c>a$, then $e_{\gcdi(m,n)}(p_i)\geq a$, which is a contradiction.
      %\par Thus for each prime divisor $p_i$ of $m$ or $n$, we obtain $e_{\gcdi(mk_2+nk_1,mn)}(p_i)=e_{\gcdi(m,n)}(p_i)$. This completes the proof.
 %It is enough to show that $\gcdi(mk_2+nk_1, mn)$ divides $ \gcdi(m,n) $. 
    \par For the converse part, if $m=n$ is an odd prime then it is clear from Remark \ref{eqdiv}. Then we may assume that $|e_m(p_i)-e_n(p_i)|\geq 1$ for all primes $p_i$ in the prime factorization of $m$ or $n$. By Remark \ref{eqdiv}, %it is enough to show that $\gcdi(mk_2+nk_1,mn)$ divides $\gcdi(m,n)$. Moreover, the prime divisors set for $\gcdi(mk_2+nk_1, mn)$ and $\gcdi(m,n)$ are the same. We now claim that
    it is enough to show that $e_{\gcdi(mk_2+nk_1,mn)}(p_i)=e_{\gcdi(m,n)}(p_i)$ for every prime $ p_i$ divisor of $\gcdi(mk_2+nk_1,mn)$. If possible let $p_i$ be such a prime that  $e_{\gcdi(mk_2+nk_1,mn)}(p_i) > e_{\gcdi(m,n)}(p_i)$. For simplicity suppose $e_{\gcdi(mk_2+nk_1,mn)}(p_i)=c$ and $ e_{\gcdi(m,n)}(p_i)=b$, thus $ c> b$. Then $e_m(p_i)\geq b$ and $e_n(p_i)\geq b$. % $p_i^{b}$ divides $ m $ and $n$ both.  
      Again by the hypothesis without loss of generality, we may assume that $e_m(p_i) > e_n(p_i)$. Thus $ e_n(p_i)= b$ as $ e_{\gcdi(m,n)}(p_i)=b$ and suppose $e_m(p_i)=a$ with $ a >b$.  If $c>a$, then by the definition of $c$, $p_i^a$ divides $\gcdi(mk_2+nk_1,mn)$. Also as $e_m(p_i)=a$ and $ \gcdi(k_1,m)=1$ it follows that  $e_{\gcdi(m,n)}(p_i)\geq a$, which is a contradiction. Thus $a\geq c$, then $e_m(p_i)\geq c$ and hence by the preceding argument $e_{n}(p_i)\geq c$. This implies $e_{\gcdi(m,n)}(p_i)\geq c$, a contradiction and hence $b=c$. %Thus for each prime divisor $p_i$ of $m$ or $n$, we obtain $e_{\gcdi(mk_2+nk_1,mn)}(p_i)=e_{\gcdi(m,n)}(p_i)$. 
      Thus we have completed the proof.
    \end{proof}
Now we will establish the necessary and sufficient conditions for $\ord(pq)$ to always be non-maximum in $\Gamma$.  
\begin{prop}
    The $\ord(pq)$ is less than $l$ for all $p,q\in \mathbb{k}^*$ with $\ord(p)=m$ and $\ord(q)=n$ if and only if $e_m(2)=e_n(2)\geq 1$. 
\end{prop}
\begin{proof}
Suppose that $\ord(pq)$ is less than $l$ for all $p,q\in \mathbb{k}^*$ with $\ord(p)=m$ and $\ord(q)=n$. Let us define the set 
\[S:=\{p_i~|~e_m(p_i)=e_n(p_i)\geq 1\}.\] 
Then by our assumption, it follows from Proposition \ref{suff} that the set $S$ is nonempty. We claim that $2\in S$. Assume (for contradiction) that %an odd prime $p_1\in S$ but 
$2\notin S$. We now find a pair $(p,q)$ with $\ord(p)=m$ and $\ord(q)=n$ such that $\ord(pq)=l$. We set $d=\gcdi(m,n)$ and $n=s_1d$ and $m=s_2d$. It is clear that $\gcdi(s_1,s_2)=1$ and $l=s_1s_2d$. Denote $h:=\prod\limits_{p_i\in S}p_i$. As $S$ is nonempty, the number $h$ exists (an odd number). It is easy to observe that $h$ divides $d$ and $\gcdi(h,s_1)=\gcdi(h,s_2)=1$. For each $p_i\in S$, consider the congruence relations:
\begin{itemize}
    \item [(1)] if $s_1\equiv s_2~(\mo p_i)$, then $x\equiv 1~(\mo p_i)$ and $y\equiv 1~(\mo p_i)$, else\\ if $s_1\not\equiv s_2~(\mo p_i)$, then $x\equiv 1~(\mo p_i)$ and $y\equiv -1~(\mo p_i)$;   
    \item [(2)] $x\equiv 1~(\mo s_1s_2)$ and $y\equiv 1~(\mo s_1s_2)$.
\end{itemize}
Since all $p_i\in S$ are relatively prime to each other and $s_1s_2$, the Chinese remainder theorem shows there are positive integers $x:=k_1$ and $y:=k_2$
which satisfy the congruence (1) and (2) modulo $hs_1s_2$. Since all the prime factors of $m$ and $n$ divide $hs_1s_2$, therefore from congruence (1) and (2) we obtain that $\gcdi(k_1,m)=\gcdi(k_2,n)=1$. This implies $\gcdi(s_1,s_2k_2)=\gcdi(s_2,s_1k_1)=1$. Therefore $\gcdi(s_1k_1+s_2k_2,l)=\gcdi(s_1k_1+s_2k_2,t)$ where $t=\prod\limits_{p_i\in S}p_i^{e_m(p_i)}$ is the largest divisor of $d$ with $\gcdi(s_1,t)=\gcdi(s_2,t)=1$. Then either case of congruence (1) implies that $s_1k_1+s_2k_2\not\equiv 0~(\mo p_i)$ for each $p_i\in S$. Thus it follows that $\gcdi(s_1k_1+s_2k_2,l)=1$ for some $k_1,k_2$ satisfying $\gcdi(k_1,m)=\gcdi(k_2,n)=1$. For this choice, take $p=g^{s_1k_1}$ and $q=g^{s_2k_2}$ so that $\ord(p)=m$ and $\ord(q)=n$ and $\ord(pq)=l$, a contradiction to the hypothesis.

\par Conversely suppose that $e_m(2)=e_n(2)\geq 1$. Then $s_1=\frac{n}{\gcdi(m,n)}$ and $s_2=\frac{m}{\gcdi(m,n)}$ are both odd numbers. Moreover the $k_1,k_2$ are odd numbers when $\gcdi(k_1,m)=1$ and $\gcdi(k_2,n)=1$. Thus the number $s_1k_1+s_2k_2$ is even for all such $k_1,k_2$. Hence by (\ref{ordpq}), the $\ord(pq)$ is less than $\lcmu(m,n)$ for all $p,q\in \mathbb{k}^*$ with $\ord(p)=m$ and $\ord(q)=n$.
\end{proof}
\section{The Inverse Parameter Case}\label{sec9}
Throughout this section, fix $p=q^{-1}\neq 1$ and set $\mathcal{H}'_q=\mathcal{H}_{q^{-1},q}$. Let us recall that $\mathcal{H}'_q$ is an associative $\mathbb{k}$-algebra generated by $x,y,z$ with the relations
\[zx=qxz,\ zy=q^{-1}yz,\ yx-qxy=z.\] 
As $p=q^{-1}$, we can utilize the second equality of (\ref{pqnum}) to obtain $[k]_{p,q}=kq^{k-1}$.
%As $p=q^{-1}$, then from the second equality of (\ref{pqnum}), we obtain $[k]_{p,q}=kq^{k-1}$. %The identities in Lemma \ref{identity} hold true for any $p,q\in\mathbb{k}^*$.
Suppose $q$ is a primitive $n$-th root of unity in $\mathbb{k}$. Similar to Subsection \ref{ztorsion}, every $z$-torsion simple modules over $\mathcal{H}'_q$ becomes a simple modules over the quantum plane $\mathcal{O}_{{q}}(\mathbb{k}^2)=\langle y,x~|~yx=qxy\rangle$.
%Then $\mathcal{O}_{{q}}(\mathbb{k}^2)$ is a PI ring with PI degree $n$ and simple modules are classified in \cite{moh}. 
Thus each $z$-torsion simple module $\mathcal{N}$ over $\mathcal{H}'_q$ is finite $\mathbb{k}$-dimensional with $\mathbb{k}$-dimension is either $1$ (if $\mathcal{N}$ is $x$-torsion or $y$-torsion) or $n$ (if $\mathcal{N}$ is $x,y$-torsionfree). It only remains to classify $z$-torsionfree simple modules. %Observe that $(0)$ is a prime ideal in $\mathcal{H}'_q$. 
By our root of unity assumption, $\ch(\mathbb{k})\notdivides n$. Now we consider the following cases based on the $\ch(\mathbb{k})$.\\
% $\gcdi(\mathrm{p},n)=1$ and 
\noindent\textsc{Case 1:} Suppose that $\ch(\mathbb{k})=\mathrm{p}>0$. Then we can easily verify from the defining relations and the identities in Lemma \ref{identity} that the elements $z^n,x^{n\mathrm{p}},y^{n\mathrm{p}}$ are in the center of $\mathcal{H}'_q$. Thus $\mathcal{H}'_q$ is a finitely generated module over a central subalgebra $\mathbb{k}[z^n,x^{n\mathrm{p}},y^{n\mathrm{p}}]$. Hence $\mathcal{H}'_q$ becomes a PI algebra by \cite[Corollary 13.1.13]{mcr}. Consequently, according to Proposition \ref{main}, it is quite clear that each simple $\mathcal{H}'_{q}$-module is finite-dimensional and can have dimension at most $\pideg(\mathcal{H}'_{q})$. Similar to Subsection \ref{torfree}, in this case, the $z$-torsionfree simple $\mathcal{H}'_q$-modules can be classified by taking a common eigenvector of the commuting operators $z,xy,x^{n\mathrm{p}}$ and $y^{n\mathrm{p}}$. Here the possible $\mathbb{k}$-dimensions of such a simple module $\mathcal{N}$ is either $n\mathrm{p}$ (if $\mathcal{N}$ is $x$-torsionfree or $y$-torsionfree) or $\mathrm{p}$ (if $\mathcal{N}$ is $x,y$-torsion). Thus in view of Proposition \ref{main}, the PI degree of $\mathcal{H}'_q$ is given by $\pideg(\mathcal{H}'_q)=n\mathrm{p}$. \\
\noindent\textsc{Case 2:} Suppose that $\ch(\mathbb{k})=0$. For $\lambda\in\mathbb{k}^*$, define the action of $\mathcal{H}'_{q}$ on the vector space $\mathbb{k}[t]$ by
\[t^kz=\lambda q^kt^k,\ t^ky=t^{k+1},\ t^kx=\begin{cases}
    kq^{k-1}\lambda t^{k-1},& k\geq 1\\
    0,& k=0.
\end{cases}\]
It is easy to verify that $\mathbb{k}[t]$ is a simple module over $\mathcal{H}'_{q}$. Moreover the annihilator of $\mathbb{k}[t]$ in $\mathcal{H}'_{q}$ is $(0)$, as every nonzero prime ideal contains $z$ (see \cite[Proposition 4.1]{jg}). Hence $\mathcal{H}'_{q}$ is a primitive algebra when $\ch(\mathbb{k})=0$. In this case, there is no finite-dimensional $z$-torsionfree simple $\mathcal{H}'_{q}$-module. It is important to note that Case 2 remains independent of the root of unity assumption. 
\par In particular combining both cases, we have established the following:
\begin{prop}
Suppose $q$ is a root of unity. Then $\mathcal{H}'_{q}$ is a PI algebra if and only if $\ch(\mathbb{k})>0$. 
\end{prop}
%n this case, each -torsionfree simple module is infinite k-dimensional. Thus we obtain the following Proposition 9.1. Suppose is a root of unity. Then is a PI ring if and only if ch(k) > 0.
%If $\ch(k)=0$ and $q$ is a root of unity, the factor algebra $\mathcal{H}'_q/\langle z\rangle\cong \mathcal{O}_{{q}}(\mathbb{k}^2)$ is a PI ring and
\section*{Acknowledgments}
The second author would like to express his sincere gratitude to the National Board of Higher Mathematics, Department of Atomic Energy, Government of India for providing funding support for research work. %The authors would also like to extend their heartfelt thanks to the anonymous referee for their meticulous review of the paper and for providing insightful feedback and suggestions that helped to enhance the overall quality of the manuscript.
\section*{Data Availability Statement}
Data sharing does not apply to this article as no datasets were generated or analyzed during the current study.
%\section*{Declarations}
%The authors have received funding from the National Board of Higher Mathematics, Department of Atomic Energy, Government of India. The authors have no other interests to disclose.

\end{document}